\documentclass{article}
\pdfoutput=1
\usepackage[utf8]{inputenc}
\usepackage{enumerate}
\usepackage{amsfonts,amsmath,amssymb,amsthm,bbm,bm}
\usepackage{mathrsfs}
\usepackage{mathtools}
\usepackage{xcolor}
\usepackage{soul}
\setstcolor{red}

\usepackage{graphicx}
\usepackage{caption}
\usepackage{subcaption}
\usepackage{float}
\usepackage{enumitem}

\usepackage[style=alphabetic,backend=biber, sorting = nyt, maxbibnames=99]{biblatex}
\usepackage{sepfootnotes}
\usepackage[hyperfootnotes=false]{hyperref}
\hypersetup{colorlinks}
\usepackage[colorinlistoftodos]{todonotes}

\usepackage{thmtools}
\usepackage{hyperref}
\usepackage[capitalise]{cleveref}
\hypersetup{
  colorlinks=false
}

\newtheorem{theorem}{Theorem}[section]

\newtheorem{lemma}[theorem]{Lemma}
\newtheorem{proposition}[theorem]{Proposition}

\theoremstyle{remark}
\newtheorem{remark}[theorem]{Remark}

\theoremstyle{definition}
\newtheorem{definition}[theorem]{Definition}
\newtheorem{example}[theorem]{Example}

\numberwithin{equation}{section}

\newcommand{\E}{\mathbb{E}}
\newcommand{\R}{\mathbb{R}}
\newcommand{\N}{\mathbb{N}}

\newcommand{\sP}{\mathbb{P}}
\newcommand{\calP}{\mathcal{P}}

\newcommand{\calS}{\mathcal{S}}

\newcommand{\bI}{\mathbf{I}}

\newcommand{\bt}[1]{\mathbf{#1}}
\newcommand{\btp}[1]{\mathbf{#1'}}

\def\calN{{\mathcal{N}}}
\def\calD{{\mathcal{D}}}

\newcommand{\calJ}{\mathcal{J}}

\newcommand{\W}{\mathcal{W}}
\newcommand{\AW}{\mathcal{AW}}

\newcommand{\DKL}{\calD_{\mathrm{KL}}}

\newcommand{\cpl}{\mathrm{Cpl}}
\newcommand{\bccpl}{\mathrm{Cpl}_{\mathrm{bc}}}

\newcommand{\tr}{\mathrm{tr}}
\newcommand{\diag}{\mathrm{diag}}

\newcommand{\sign}{\mathrm{sign}}

\title{Entropic adapted Wasserstein distance on Gaussians}
\author{
Beatrice Acciaio\thanks{ETH Z\"{u}rich, Department of Mathematics, Switzerland.
\emph{beatrice.acciaio@math.ethz.ch}
}~~~
Songyan Hou\thanks{ETH Z\"{u}rich, Department of Mathematics, Switzerland.
\emph{songyan.hou@math.ethz.ch}
}~~~
Gudmund Pammer\thanks{TU Graz, Institute of Statistics, Austria. \emph{gudmund.pammer@tugraz.at}
}
}

\date{\today}
\begin{document}
\maketitle

\begin{abstract}
    The adapted Wasserstein ($\AW$) distance is a metric for quantifying distributional uncertainty and assessing the sensitivity of stochastic optimization problems on time series data. 
    A computationally efficient alternative to it, is provided by the entropically regularized $\AW$-distance.
    Suffering from similar shortcomings as classical optimal transport, there are only few explicitly known solutions to those distances.
    Recently, Gunasingam--Wong \cite{gunasingam2024adapted} provided a closed-form representation of the $\AW$-distance between real-valued stochastic processes with Gaussian laws. In this paper, we extend their work in two directions, by considering multidimensional ($\R^d$-valued) stochastic processes with Gaussian laws and including the entropic regularization.
    In both settings, we provide closed-form solutions.
    \\

\end{abstract}

\section{Introduction} \label{sec:intro}
The Wasserstein distance is a ubiquitous concept with applications across numerous fields, such as statistics, economics, biology, and machine learning. Its rise was crucially amplified by the introduction of entropic regularization, which traces back to Schr\"{o}dinger's work in 1931 \cite{ISch1931}; see also \cite{leonard2013survey} for a survey on the connection with optimal transport. In the current paper, extending the work of Gunasingam--Wong \cite{gunasingam2024adapted}, we provide a closed-form solution of the entropic adapted Wasserstein distance between stochastic processes with Gaussian laws. 

For probabilities $\mu,\nu$ on $\R^d$, the entropic 2-Wasserstein distance $\W_{2,\lambda}$ is defined by
\begin{equation}
    \label{eq:W.reg}
    \W_{2,\lambda}^2(\mu,\nu) 
    \coloneqq \inf_{\pi \in \cpl(\mu,\nu)} \int |x-y|^2 \, d\pi + \lambda \mathcal D_{\rm KL}(\pi | \mu \otimes \nu),
\end{equation}
where $\lambda \ge 0$ is the regularization parameter, $\mathcal D_{\rm KL}$ is the Kullback-Leibler divergence (or relative entropy), and $\cpl(\mu,\nu)$ denotes the set of all probabilities on $\R^d \times \R^d$ with first marginal $\mu$ and second marginal $\nu$. Clearly, for $\lambda=0$ we recover the classical 2-Wasserstein distance  $ \W_{2}(\mu,\nu)$.
Building on transport theory, adapted optimal transport has been introduced to deal with distributions in dynamic settings, where the time component and the flow of information play crucial roles.
Remarkably, this
provides a robust framework for distributional uncertainty and sensitivity of stochastic optimization problems of time series data; see e.g.\ \cite{backhoff2020adapted, pflug2014multistage, backhoff2017causal, acciaio2022quantitative, Bartl2023SensitivityOM}.

For two laws of discrete-time stochastic processes $\mu,\nu$ on $\R^{dT}$, where $T$ is the number of time steps, the entropic adapted 2-Wasserstein distance $\AW_{2,\lambda}$ is defined by
\begin{equation}
    \label{eq:AW.reg}
    \AW_{2,\lambda}^2(\mu,\nu) \coloneqq
    \inf_{\pi \in \bccpl(\mu,\nu)} \int |x-y|^2 \, d\pi + \lambda \mathcal D_{\rm KL}(\pi | \mu \otimes \nu),
\end{equation}
where $\bccpl(\mu,\nu)$ denotes the set of bi-causal couplings between $\mu$ and $\nu$; see Definition~\ref{def:biccpl} below. 
For $\lambda=0$ this boils down to the adapted 2-Wasserstein distance, denoted by $\AW_{2}(\mu,\nu)$.

As for the classical 
Wasserstein distance, in the adapted setting it is notoriously difficult to provide closed form solutions of the distance and the optimal couplings.
Our main results are Theorems~\ref{thm.EAW_dT} and \ref{thm.EAW_dT_sol} in \Cref{sec.eawd} below, which extend \cite{gunasingam2024adapted} to the multi-dimensional entropically regularized setting.
For comparison and simplicity of exposition, we present them here for the case $d=1$ (real-valued processes).
For $\lambda \ge 0$, we define the function $f_\lambda(x)$ as $\sign(x)\Big(\sqrt{(\frac{\lambda}{4|x|})^2 + 1} - \frac{\lambda}{4|x|}\Big)$ if $x\neq 0$ and $0$ otherwise. {\color{black} Note that $f_\lambda(s)$ is the unique optimizer of $\max_{x\in (-1,1)}2sx + \frac{\lambda}{2}\log(1-x^2)$ when $\lambda > 0$, and otherwise the sign function.}

\begin{theorem} \label{thm:AW.reg}
    Let $\mu = \mathcal N(a,A)$ and $\nu = \mathcal N(b,B)$ be non-degenerate Gaussians on $\R^T$, whose covariance matrices have Cholesky decompositions $A = LL^\top$ and $B = MM^\top$. Then
    \begin{equation}
        \label{eq:thm.AW.reg}
        \AW^2_{2,\lambda}(\mu,\nu) = |a - b|^2 +
        \tr\big(A + B\big) - 2\tr\big(P_\lambda N \big) - \frac{\lambda}{2} \log \det(I - P_\lambda^2) ,
    \end{equation}
    where $P_\lambda = f_\lambda(N)$\footnote{\color{black} For $f\colon \R\to \R$ and $M\in\R^{d \times d}$, we denote by $f(M)$ the element-wise function application of $f$ on $M$.} and $N = \diag([(M^\top L)_{tt}]_{t=1}^{T})$.
\end{theorem}

\begin{theorem}\label{thm:AW.reg_sol}
    In the setting of \Cref{thm:AW.reg}, for any diagonal matrix $P\in\R^{T\times T}$ s.t. $P_{tt} \in [-1,1]$, $t \leq T$, define
    \begin{equation}
        \label{eq:thm.AW.reg.optimizer}
        \pi_P \coloneqq
        \mathcal N\bigg(
        \begin{bmatrix}
            a \\
            b
        \end{bmatrix},
        \begin{bmatrix}        
            LL^\top & LP M^\top \\
            MP^\top L^\top & MM^\top
        \end{bmatrix}\bigg).
    \end{equation}
    Then we have:
    \begin{enumerate}[label = (\roman*)]
        \item If $\lambda > 0$, let $P = P_\lambda$. Then $\pi^*=\pi_P$ is the unique optimizer of $\AW_{2,\lambda}(\mu,\nu)$.
        \item If $\lambda = 0$, let $P_{tt} = (P_\lambda)_{tt}=\sign(N_{tt})$ if $N_{tt} \neq 0$, $t\leq T$. Then $\pi^*=\pi_P$ is an optimizer of $\AW_{2,\lambda}(\mu,\nu)$ and all Gaussian optimizers have such a representation. Moreover, $\pi^*$ is the unique optimizer if and only if $N$ is invertible.
    \end{enumerate}
\end{theorem}

In what follows, we draw comparisons between the results obtained here for the $\AW_{2,\lambda}$-distance 
and those contained in previous literature, especially regarding regularization, dimensionality, and adaptedness.

\medskip

\noindent $\diamond$ \emph{Regularized vs unregularized adapted Wasserstein distance.}
If $\lambda = 0$, $f_\lambda(x)x = |x|$ for all $x \in \R$. 
Then $\tr(f_\lambda(N)N) = \tr(|N|) = \tr(|M^\top L|) =  \tr(|L^\top M|)$ and we recover the closed-form solution of $\AW_2$ for Gaussians on $\R^T$ from  \cite[Theorem~1.1]{gunasingam2024adapted}. 
In the setting of \Cref{thm:AW.reg}, this reads as
\begin{equation}
        \AW^2_{2}(\mu,\nu) = |a - b|^2 +
        \tr\big(A + B) - 2\tr(|L^\top M|).
\end{equation}
When the matrix $N$ (c.f.\ \Cref{thm:AW.reg}) is not invertible, then $\AW_{2}(\mu,\nu)$ does not have a unique optimizer.
In this case there are also non-Gaussian optimizers; see Example~\ref{ex:1}. 
However, $\AW_{2,\lambda}(\mu,\nu)$ always has a unique, Gaussian optimizer when $\lambda > 0$, thanks to the strict convexity of the Kullback-Leibler divergence.
\medskip

\noindent $\diamond$ \emph{One-dimensional vs multi-dimensional adapted Wasserstein distance.}
The multi-dimensional generalization of Theorems~\ref{thm:AW.reg}-\ref{thm:AW.reg_sol}, that is, when $\mu$ and $\nu$ are Gaussian distributions on $\R^{dT}$ for $d \ge 1$, can be found in Theorems~\ref{thm.EAW_dT} and \ref{thm.EAW_dT_sol} below.
When $d = 1$, it was already observed in \cite{gunasingam2024adapted} that the Knothe-Rosenblatt coupling is an optimizer for $\AW_{2}(\mu,\nu)$ if and only if the diagonal entries of $L^\top M$ are non-negative. In the present setting, the Knothe-Rosenblatt coupling corresponds to the coupling defined in \eqref{eq:thm.AW.reg.optimizer} with $P = I$. For a detailed discussion of the Knothe-Rosenblatt coupling and its induced distance, we refer to \cite{beiglbock2023knothe}, where two generalizations are introduced{\color{black}; one based on defining the quantile process on $\R^{dT}$, and the second based on defining the triangular optimal couplings; see Section~1.7 in \cite{beiglbock2023knothe} for detailed definitions.} The second one is referred to as adapted Brenier coupling in \cite{gunasingam2024adapted}. Moreover, by \cite[Corollary 4.6]{gunasingam2024adapted} the Knothe-Rosenblatt distance coincides with $\AW_2$ locally at non-degenerate Gaussians.
On the other hand, when $d > 1$, there is no direct analogue to the Knothe-Rosenblatt coupling and the structure of $\AW_2$-optimal Gaussian couplings becomes more complex as the $dT \times dT$-analogue to the $T \times T$-diagonal matrix $P_\lambda$ is given by a block diagonal matrix.

\medskip

\noindent $\diamond$ \emph{Adapted vs classical (regularized) Wasserstein distance.}
The entropic 2-Wasserstein distance between non-degenerate Gaussians has a well-known closed-form solution, and the unique optimal coupling is also Gaussian; see \cite{mallasto2022entropy, del2020statistical, bojilov2016matching, janati2020entropic}.
In the setting of \Cref{thm:AW.reg}, this gives
\begin{equation} \label{eq:ewd}
    \W^2_{2,\lambda}(\mu, \nu) = |a - b |^2 + \tr(A+B) - 2\tr(C_\lambda) -\frac{\lambda}{2}\log\det(I - K_\lambda K_\lambda^\top),
\end{equation}
where the matrices $K_\lambda$ and $C_\lambda$ are given by
\begin{equation}
\label{eq:ewd_opl_cpl.0}
    K_\lambda = A^{-\frac{1}{2}}C_\lambda B^{-\frac{1}{2}},
    \quad 
    C_\lambda \coloneqq \frac{1}{2}A^{\frac{1}{2}}\Big(4A^{\frac{1}{2}}B A^{\frac{1}{2}} + \frac{\lambda^2}{4}I\Big)^{\frac{1}{2}} A^{-\frac{1}{2}} - \frac{\lambda}{4} I,
\end{equation}
and the unique optimal coupling is $\pi^* = 
\calN\bigg(\begin{bmatrix}
a\\
b
\end{bmatrix}, \begin{bmatrix}
    A & C_\lambda\\
    C_\lambda^\top & B
\end{bmatrix}\bigg)$.
When $T=1$, the optimal transport problem coincides with the adapted optimal transport problem, and in this case we have $\AW_{2,\lambda} = \W_{2,\lambda}$ and $\pi^*$ (given as above) coincides with the optimal coupling that is characterized in \Cref{thm.EAW_dT_sol}. 
On the other hand, when $T \geq 2$, the values of $\AW_{2,\lambda}$ and $\W_{2,\lambda}$ typically differ from each other.
From the formula of $\AW_{2,\lambda}$ and $\W_{2,\lambda}$, we obtain an entropic trace inequality; see \Cref{subsec.Aad_vs_class} for details. 
When $\lambda=0$, $\W_{2}$ has a unique, Gaussian optimal coupling between non-degenerate Gaussians. Moreover, the optimal coupling is a Monge optimal coupling with corresponding Monge map given by the linear transformation $x \mapsto b + A^{-\frac{1}{2}}(A^{\frac{1}{2}}B A^{\frac{1}{2}})^{\frac{1}{2}}A^{-\frac{1}{2}}(x-a)$. By Monge optimal coupling, we mean an optimal coupling supported on the graph of a map, which is in turn referred to as Monge map. On the other hand, when $\lambda =0$, $\AW_{2}$ may not have a unique optimal coupling. For $\AW_{2}$, there always exists a Monge optimal coupling, while there may also exist optimal couplings that are not of Monge type; see \Cref{subsec.Monge} for details.

\medskip

In this paper, we focus on the setting of non-degenerate Gaussian distributions, a common assumption in the study of entropic Wasserstein distances between Gaussians (see \cite{mallasto2022entropy}). Our closed-form expression for the entropic adapted Wasserstein distance between Gaussians is fundamentally based on the Cholesky decomposition. 
For degenerate Gaussians, this decomposition is not unique --—different Cholesky decompositions correspond to stochastic processes with the same law but different filtrations. The study of the degenerate case thus requires the use of different tools, in particular to work in the context of filtered processes, which is left for future work.
\medskip

Other papers investigating explicit solutions of causal and bi-causal transport problems are \cite{backhoff2017causal, backhoff2022adapted, ruschendorf1985wasserstein}  studying the Knothe-Rosenblatt coupling in discrete time, \cite{backhoff2020adapted, backhoff2022adapted, bion2019wasserstein, cont2024causal, lassalle2018causal, robinson2024bicausal} for continuous time SDEs, \cite{han2023distributionally, zorzi2020optimal} solving a related optimal transport problem between Gaussians, and
\cite{ramgraber2023ensemble, baptista2024conditional} constructing triangular maps as conditional Brenier maps between Gaussians.

\paragraph{Notations.}
We regard $\R^{dT}$ as the space of $d$-dimensional discrete-time paths with $T$ time steps, $x = (x_1,\dots,x_T)$, equipped with the Euclidean norm $|\cdot|$.
To refer to the subvector $(x_s,\dots,x_t)$ of $x$ we also write $x_{s:t}$.
For $t=1,\ldots, T-1$, we use the shorthand notations $x_{\bt{t}}=x_{1:t}$ and $x_{\btp{t}}=x_{t+1:T}$, and set $x_{\bt{T}}=x$. Moreover, for $\mu\in\mathcal{P}(\R^{dT})$, we denote the $t$-th marginal of $\mu$ by $\mu_t$, the up-to-time-$t$ marginal of $\mu$ by $\mu_{\bt{t}}$, and the kernel (disintegration) of $\mu$ w.r.t. $x_{\bt{t}}$ by $\mu_{x_{\bt{t}}}$, so that $\mu(dx_{\btp{t}}) = \int_{\R^{dt}}\mu_{x_{\bt{t}}}(dx_{\btp{t}}) \mu_{\bt{t}}(dx_{\bt{t}})$.
To denote the $x_s$-marginal of $\mu_{x_{\bt{t}}}$, we use the notation $\mu_{x_{\bt{t}}}^s(dx_s) = (x_{\bt{t}'} \mapsto x_s)_\# \mu_{x_{\bt{t}}}(dx_s)$, where $\#$ designates the push-forward operation.
For notational completeness, we let $\mu_{x_{0}} = \mu$. 
For $\mu,\nu \in\mathcal{P}(\R^{dT})$, we denote their set of couplings by $\cpl(\mu,\nu)=\{\pi\in\mathcal{P}(\R^{dT}\times\R^{dT}) : \pi(dx \times \R^{dT}) = \mu(dx), \pi(\R^{dT}\times dy) = \nu(dy) \}$, 
and for every coupling $\pi$ we use the analogous notations
$\pi_{t}$, $\pi_{\bt{t}}$, $\pi_{x_{\bt{t}},y_{\bt{t}}}$, $\pi^s_{x_{\bt{t}},y_{\bt{t}}}$, $\pi_{x_{0}, y_0}$.
For any matrix $M \in \R^{dT \times dT}$ and for $s,t\in \{1,\dots, T\}$, we denote the $s$-th (resp.\ up to $s$) row and $t$-th (resp.\ up to $t$) column blocks of $M$ by $M_{s,t} \in \R^{d\times d}$ (resp.\ $M_{\bt{s},\bt{t}} \in \R^{ds\times dt}$), so that
\[
M = \begin{bmatrix}
M_{1,1}  & \cdots & M_{1,T} \\
\vdots  & \ddots & \vdots \\
M_{T,1}  & \cdots & M_{T,T}
\end{bmatrix},\quad 
M_{\bt{s},\bt{t}} = \begin{bmatrix}
M_{1,1} & \cdots & M_{1,t} \\
\vdots & \ddots & \vdots \\
M_{s,1} & \cdots & M_{s,t}
\end{bmatrix}.
\]
If $M$ is block diagonal, meaning $M_{s,t}=0$ whenever $s\neq t$, with an abuse of notation we denote $M_t = M_{t,t}$, $t \in \{1,\dots, T\}$. We equip matrices with the spectral norm, denoted by $\Vert \cdot \Vert_2$.
For $f\colon \R\to \R$ and $M\in\R^{dT \times dT}$, we denote by $f(M)$ the element-wise function application of $f$ on $M$. Similarly, for any $x \in \R^{dT}$, we adopt the convention that $x = [x_1^\top, \ldots, x_T^\top]^\top$. For the sign function, we adopt the convention $\sign(0)=0$. Finally, the \textit{Kullback–Leibler divergence} is given by $\DKL(\mu|\nu) = \int \log({d\mu}/{d\nu})d\mu$ if $\mu \ll \nu$ and $+\infty$ otherwise.

\section{Main results}\label{sec.eawd}
We start by introducing the notion of bi-causal coupling, as a coupling for which, at every time, the conditional law of the future evolution given the past is still a coupling of the respective conditional laws of the marginals.
\begin{definition}
\label{def:biccpl}
    A coupling $\pi\in\cpl(\mu,\nu)$ is called bi-causal if, for all $t=1,\ldots T-1$ and $x_\bt{t}$, $y_\bt{t} \in \R^{dt}$, it holds that $\pi_{x_\bt{t},y_\bt{t}} \in \cpl(\mu_{x_\bt{t}}, \nu_{y_\bt{t}})$. In this case, we write $\pi\in\bccpl(\mu,\nu)$.
\end{definition}
The causality constraint can be expressed in different equivalent ways, see e.g.\ \cite{backhoff2017causal,acciaio2020causal} in the context of transport, and \cite{bremaud1978changes} in the filtration enlargement framework. In particular, the bi-causality condition in the above definition corresponds to having the following conditional independence for discrete-time processes $X,Y$ with joint distribution $\pi$:
\begin{equation*}
\text{for all $t=1,\ldots T-1$},\quad Y_t \perp X\; \text{given $X_{1:t}$},\quad \text{and }\, X_t \perp Y\; \text{given $Y_{1:t}$}.
\end{equation*}
This means that the transport is done in a non-anticipative way, both from $X$ to $Y$ and from $Y$ to $X$.
In the next theorem, we characterize all Gaussian bi-causal couplings.

\begin{theorem}[Characterization of Gaussian bi-causal couplings]
\label{thm:gau_ad}
    Let $\mu = \mathcal N(a,A)$ and $\nu = \mathcal N(b,B)$ be non-degenerate Gaussians on $\R^{dT}$, whose covariance matrices have Cholesky decompositions $A = LL^\top$ and $B = MM^\top$.
    A Gaussian coupling $\pi \in \cpl(\mu,\nu)$ is bi-causal if and only if there exists a block diagonal matrix $P = \diag(P_1,\dots,P_T) \in \R^{dT \times dT}$ s.t.\ $\pi=\pi_P$, where
    \begin{equation}
    \label{eq:gauss_ansatz}
    \pi_P := \calN\Big(
    \begin{bmatrix}
        a \\ b
    \end{bmatrix},
    \begin{bmatrix}
            LL^\top & LPM^\top \\
            M P^\top L^\top &MM^\top
        \end{bmatrix}\Big),
    \end{equation}
    and $P_{t} \in \R^{d \times d}$ are contractions (i.e. $\Vert P_{t} \Vert_2 \leq 1$), for all $t\leq T$.
\end{theorem}

\begin{proof}
    If $\pi \in \bccpl(\mu,\nu)$ is Gaussian, then there exists $C\in \R^{dT \times dT}$ s.t.\ $\begin{bmatrix}
            LL^\top & C \\
            C^\top & MM^\top
        \end{bmatrix}$ is positive definite and $\pi = \calN\Big(
        \begin{bmatrix}
            a \\
            b
        \end{bmatrix},
        \begin{bmatrix}
            LL^\top & C \\
            C^\top & MM^\top
        \end{bmatrix}\Big)$.
    Let $(X,Y) \sim \pi$. 
    Since $L$ and $M$ are invertible, we can define
    \begin{equation*}
        \begin{bmatrix}
            Z^X \\
            Z^Y
        \end{bmatrix}
        \coloneqq
        \begin{bmatrix}
            L^{-1} & 0 \\
            0 & M^{-1}
        \end{bmatrix}
        \begin{bmatrix}
            X - a \\
            Y - b
        \end{bmatrix}
        \sim \calN\Big(
        \begin{bmatrix}
            0 \\
            0
        \end{bmatrix},
        \begin{bmatrix}
            I & P \\
            P^\top & I
        \end{bmatrix}\Big) \eqqcolon \pi^{Z},
    \end{equation*} 
    where $P = L^{-1}C(M^{-1})^\top$.
    By bi-causality of $\pi$ and since $\begin{bmatrix}
            L^{-1} & 0 \\
            0 & M^{-1}
        \end{bmatrix}$ is invertible and lower triangular,
    $\pi^{Z}$ is also bi-causal and Gaussian. Then, for all $s < t, s,t \leq T$, $i,j \leq d$, 
    \begin{equation*}
        (P_{s,t})_{ij} = \E[Z^X_{s,i}Z^Y_{t,j}] =\E[Z^X_{s,i}\E[Z^Y_{t,j} | Z^X_\bt{t-1},Z^Y_\bt{t-1}]] = \E[Z^X_{s,i}\E[Z^Y_{t,j} |Z^Y_\bt{t-1}]] = \E[Z^X_s \cdot 0 ] = 0,
    \end{equation*}
    where the second equality follows by the tower property, the third one by causality, and the forth one by $Z^Y \sim \calN(0,I)$. Similarly, by symmetry, we have $(P_{s,t})_{ij} =0$ for all $s,t \leq T, s > t, i,j \leq d$. Thus, $P$ is block diagonal, i.e. $P = \diag(P_1,\dots,P_T)$ with $P_t \in \R^{d \times d}$. Finally, $\begin{bmatrix}
        I & P \\
        P^\top & I
        \end{bmatrix}$
    is positive-definite, which is equivalent to say that $P$ is a contraction, by \cite[Proposition~1.3.1]{bhatia2009positive}. 
    In turn, this corresponds to $(P_1,\dots,P_T)$ being all contractions, as $P$ is block diagonal. 
    The reverse implication is shown by following the arguments above in reverse order, as these are all equivalences.
\end{proof}

In general, couplings of Gaussian marginals need not be Gaussian. So an optimal coupling of $\AW_{2,\lambda}(\mu,\nu)$ is not necessarily Gaussian. 
In particular, this is the case when the optimizer is non-unique; see also Example~\ref{ex:2} below. 
Nonetheless, it turns out that $\AW_{2,\lambda}(\mu,\nu)$ always admits a Gaussian bi-causal optimal coupling.
\begin{theorem}[Existence of Gaussian bi-causal optimizers]
\label{thm.opt_eaot_gau}
Let $\mu = \mathcal N(a,A)$ and $\nu = \mathcal N(b,B)$ be non-degenerate Gaussians on $\R^{dT}$, whose covariance matrices have Cholesky decompositions $A = LL^\top$ and $B = MM^\top$, and let $\lambda \ge 0$. Then $\AW_{2,\lambda}(\mu,\nu)$ has a Gaussian optimal coupling.
\end{theorem}
We postpone the proof of \Cref{thm.opt_eaot_gau} to \Cref{sec.proofthm} and head to our main theorems.
\begin{theorem} [Closed-form representation]
\label{thm.EAW_dT}
Let $\mu = \mathcal N(a,A)$ and $\nu = \mathcal N(b,B)$ be non-degenerate Gaussians on $\R^{dT}$, whose covariance matrices have Cholesky decompositions $A = LL^\top$ and $B = MM^\top$, and let $\lambda \ge 0$. Then
\begin{equation}
\label{eq:thm.eaw_dT.0}
    \begin{split}
        \AW_{2,\lambda}^2(\mu,\nu)
        &= |a-b|^2 + \tr(A + B) - 2\tr(D_{\lambda} S) - \frac{\lambda}{2}\log\det(I - D_{\lambda}^2),
    \end{split}
\end{equation}
where $D_{\lambda} = f_\lambda(S)$ and $S = \diag(S_1,\dots,S_T)$, with $S_{t}$ being the diagonal matrix of singular values of $(M^\top L)_{t,t}$. In particular, when $\lambda = 0$, $f_\lambda$ is the sign function, so \eqref{eq:thm.eaw_dT.0} reduces to
\begin{equation*}
    \begin{split}
        \AW_{2}^2(\mu,\nu)
        &= |a-b|^2 + \tr(A + B) - 2\tr(S).
    \end{split}
\end{equation*}
Moreover, we have $\lim_{\lambda \to 0}\AW_{2,\lambda}^2(\mu,\nu) = \AW_{2}^2(\mu,\nu)$.
\end{theorem}
\begin{theorem}[Characterization of optimizers]
\label{thm.EAW_dT_sol}
In the setting of 
\Cref{thm.EAW_dT}, let $U_t S_t V_t^\top$ be the singular value decomposition of $(M^\top L)_{t,t}$, $t=1,\dots,T$, and, for any block diagonal contraction matrix $P\in\R^{dT \times dT}$, define $\pi_P$ by \eqref{eq:gauss_ansatz}.
Then we have:
    \begin{enumerate}[label = (\roman*)]
        \item If $\lambda > 0$, let $P_t =  V_t^\top(D_{\lambda})_t U_t$, $t=1,\dots,T$.
        Then $\pi^* = \pi_P$ is the unique optimizer of $\AW_{2,\lambda}(\mu,\nu)$.
        \item If $\lambda = 0$, let $P_t = V_t^\top D_t U_t$ where $D_t$ is a contraction with $(D_{t})_{ii} = 1$ if $(S_t)_{ii} \neq 0$, $i = 1,\dots, d$, $t=1,\dots,T$. 
        Then $\pi^* = \pi_P$ is an optimizer of $\AW_{2}(\mu,\nu)$ and all Gaussian optimizers have such a representation. Moreover, $\pi^*$ is the unique optimizer if and only if $S$ is invertible.
    \end{enumerate}
\end{theorem}

\begin{proof}[Proof of Theorems~\ref{thm.EAW_dT}-\ref{thm.EAW_dT_sol}]
W.l.g.\ we assume $a=b=0$. By \Cref{thm.opt_eaot_gau}, there exists at least one optimal Gaussian bi-causal coupling. By \Cref{thm:gau_ad}, for any such coupling $\pi$, there exists a block diagonal matrix $P = \diag(P_1,\dots,P_T)$ with $\Vert P_t \Vert_2 \leq 1$ s.t.\ $\pi = \pi_P$ given by \eqref{eq:gauss_ansatz}.
First, we compute the transport cost under $\pi$, which is given by
\begin{equation}
\label{eq:thm.eaw_dT.1}
\begin{split}
    \int \Vert x - y \Vert^2 d\pi &= \tr(A+B) - 2\tr(LPM^\top) = \tr(A+B) - 2\tr(P M^\top L).
\end{split}
\end{equation}
Next, we compute the relative entropy of $\pi$ w.r.t.\ $\mu \otimes \nu$:
\begin{equation}
\label{eq:thm.eaw_dT.2}
\begin{split}
    \DKL(\pi|\mu\otimes \nu) 
     &= \frac{1}{2}\Big(\log\det \begin{bmatrix}
        LL^\top & 0 \\
        0 & MM^\top
    \end{bmatrix}- \log\det \begin{bmatrix}
        LL^\top & LPM^\top \\
        MP^\top L^\top & MM^\top
    \end{bmatrix}\Big)\\
    &= \frac{1}{2}\Big(\log\det \begin{bmatrix}
        I & 0 \\
        0 & I
    \end{bmatrix}- \log\det \begin{bmatrix}
        I & P \\
        P^\top & I\\
    \end{bmatrix}\Big)
    = -\frac{1}{2}\log\det(I - PP^\top).
\end{split}
\end{equation}
Combining \eqref{eq:thm.eaw_dT.1} and \eqref{eq:thm.eaw_dT.2}, we arrive at the total cost
\begin{equation}
\label{eq:thm.eaw_dT.3}
\begin{split}
  \hspace{-0.1cm}  \int \Vert x - y \Vert^2 d\pi + \lambda \DKL(\pi|\mu\otimes \nu) &= \tr(A+B) - 2\tr(P M^\top L) -\frac{\lambda}{2}\log\det(I - PP^\top)\\
    &= \tr(A+B) - \sum_{t=1}^{T}2\tr(P_t N_t) - \frac{\lambda}{2}\log\det(I - P_tP_t^\top),
\end{split}
\end{equation}
where we define $N_t \coloneqq (M^\top L)_{t,t}$, $t = 1,\dots,T$. 
Clearly, to minimize the cost in \eqref{eq:thm.eaw_dT.3}, we only need to equivalently consider, for all $t=1,\dots,T$, the maximization problem
\begin{equation}
\label{eq:thm.eaw_dT.3.5}
    \begin{split}
        \calJ_t \coloneqq \max_{\Vert P_t \Vert_2 \leq 1} 2\tr(P_t N_t) + \frac{\lambda}{2}\log\det(I - P_t P_t^\top).
    \end{split}
\end{equation}
Let $U_t S_t V_t^\top = N_t$ be the singular value decomposition of $N_t$, where $U_t, V_t$ are orthogonal and $S_t$ is diagonal. We define $D_{t}\coloneqq V_t^\top P_t U_t$. Notice that $\Vert P_t \Vert_2 \leq 1$ iff $\Vert V_t^\top P_t U_t \Vert_2 \leq 1$ and $\det(I - P_t P_t^\top) = \det(I - D_t D_t^\top)$. So, $\calJ_t$ can be equivalently expressed using $D_t$, as
\begin{align}
        \calJ_t &= \max_{\Vert P_t \Vert_2 \leq 1} 2\tr(P_t U_t S_t V_t^\top) + \frac{\lambda}{2}\log\det(I - P_t P_t^\top) \nonumber\\
        &= \max_{\Vert D_{t}\Vert_2 \leq 1} 2\tr(D_{t}S_t) + \frac{\lambda}{2}\log\det(I - D_{t}D_{t}^\top).
        \label{eq:thm.eaw_dT.4}
\end{align}
We write $s_{t,i} \coloneqq (S_t)_{ii} \geq 0$ and denote by $d_{t,i} \coloneqq \sigma_i(D_t)$ the $i$-th singular value of $D_t$, $i = 1,\dots, d$.
Thus, $(1-d_{t,1}^2,\dots,1-d_{t,d}^2)$ are the eigenvalues of $I - D_t D_t^\top$, and we have $\det(I - P_t P_t^\top) = \prod_{i=1}^{d}(1-d_{t,i}^2)$
Moreover, $\Vert D_t \Vert_2 \leq 1$ yields $d_{t,i}\in [0,1]$. 
By von-Neumann's trace inequality\footnote{$\tr(AB) \leq \sum_{i=1}^{d}\sigma_i(A) \sigma_i(B)$ where $\sigma_i(A)$ and $\sigma_i(B)$ are the $i$-th singular values of $A$ and $B$.}, we have
\begin{equation}
\label{eq:thm.eaw_dT.4.5}
\begin{split}
        \calJ_t &= \max_{\Vert D_t \Vert_2 \leq 1} 2\tr(D_t S_t ) + \frac{\lambda}{2}\sum_{i=1}^{d}\log(1-d_{t,i}^2)\\
        &\leq \max_{\Vert D_t \Vert_2 \leq 1} 2\sum_{i=1}^{d}s_{t,i} d_{t,i} + \frac{\lambda}{2}\sum_{i=1}^{d}\log(1-d_{t,i}^2) \leq \sum_{i=1}^{d}\bar{\calJ}_{t,i},
\end{split}
\end{equation}
where $\bar{\calJ}_{t,i} \coloneqq \max_{d_{t,i}\in[0,1]}2s_{t,i}d_{t,i} + \frac{\lambda}{2}\log(1-d_{t,i}^2)$.
Observe that $d^*_{t,i} \coloneqq f_\lambda(s_{t,i})$ maximizes $\bar{\calJ}_{t,i}$.
By defining $D_{\lambda,t} = {\color{black}\diag(d^*_{t,1},\ldots, d^*_{t,d})} = f_\lambda(S_t)$ and taking $D_t = D_{\lambda, t}$, we see that the inequalities in \eqref{eq:thm.eaw_dT.4.5} become equalities. 
Therefore, $D_{\lambda,t}$ is an optimizer of $\calJ_t$ and 
\begin{equation}
\label{eq:thm.eaw_dT.5}
\calJ_t = 2\tr(D_{\lambda,t} S_t) + \frac{\lambda}{2}\log \det(I - D_{\lambda,t}^2).
\end{equation}
When $\lambda > 0$, $D_{\lambda,t}$ is the unique optimizer by strict convexity. 
When $\lambda = 0$, we have that $\calJ_t = 2 \sum_{i=1}^{d}s_{t,i}$. 
Therefore, if $D_t$ is optimal, then $\tr(D_t S_t) = \sum_{i=1}^{d} (D_t)_{ii} s_{t,i} = \sum_{i=1}^{d}s_{t,i}$. Also, notice that $\Vert D_t \Vert_2 \leq 1$ implies $(D_t)_{ii} \in [-1,1]$. This means that, since $s_{t,i} \geq 0$, we get $(D_t)_{ii} = 1$ whenever $s_{t,i} > 0$. 
Therefore, we have shown that, when $\lambda = 0$, 
\[
\text{ $D_t$ is an optimizer of $\calJ_t \Longleftrightarrow D_t$ is a contraction and $(D_t)_{ii} = 1$ if $(S_t)_{ii} > 0$, $i = 1,\dots,d$.}
\]
Combining \eqref{eq:thm.eaw_dT.3}, \eqref{eq:thm.eaw_dT.3.5} and \eqref{eq:thm.eaw_dT.5}, we get
\begin{equation*}
    \begin{split}
        \AW_{2,\lambda}^2(\mu,\nu) &= \tr(A + B) - \sum_{t=1}^{T} 2\tr(D_{\lambda,t} S_t) + \frac{\lambda}{2}\log\det(I - D_{\lambda,t}^2)\\
        &= \tr(A + B) - 2\tr(D_{\lambda}S) - \frac{\lambda}{2}\log\det(I - D_{\lambda}^2),
    \end{split}
\end{equation*}
where $D_\lambda = f_\lambda(S)$, $S = {\color{black}\diag(S_1,\ldots,S_T)}$. 
Since $D_{t}= V_t^\top P_t U_t$, all Gaussian optimizers are of the form $\pi = \pi_P$ given by \eqref{eq:gauss_ansatz},
where $P$ is block diagonal with $P_t = V_t D_t U_t^\top$, and $D_t$ is an optimizer of $\calJ_t$.
\end{proof}

\subsection{Adapted vs classical (regularized) Wasserstein distance}
\label{subsec.Aad_vs_class}
When $T=1$, adapted optimal transport coincides with classical optimal transport. 
So, $\AW_{2,\lambda} = \W_{2,\lambda}$ and the problems share the same optimizers. 
In particular, the unique optimal coupling of $\W_{2,\lambda}$ between two non-degenerate Gaussians is also characterized by \Cref{thm.EAW_dT_sol}. 
To see this, recall that $\W_{2,\lambda}$ has a unique, Gaussian optimal coupling $\pi^*$:
\begin{equation}
\label{eq:ewd.optimizer}
    \pi^* = 
    \calN\bigg(\begin{bmatrix}
    a\\
    b
    \end{bmatrix}, \begin{bmatrix}
        A & C_\lambda\\
        C_\lambda^\top & B
    \end{bmatrix}\bigg), \quad C_\lambda \coloneqq \frac{1}{2}A^{\frac{1}{2}}\Big(4A^{\frac{1}{2}}B A^{\frac{1}{2}} + \frac{\lambda^2}{4}I\Big)^{\frac{1}{2}} A^{-\frac{1}{2}} - \frac{\lambda}{4} I.
\end{equation}
Let $L=H_A V_A^\top$ and $M = H_B U_B^\top$ be the polar decompositions of $L$ and $M$, where $H_A, H_B$ are positive definite and $U_A, U_B$ are orthogonal.
Then $A = LL^\top = H_A^2$ and $B = MM^\top = H_B^2$. Since $A$ and $B$ are both positive definite, they have a unique square root matrix. 
Thus, $H_A = A^\frac{1}{2}$ and $H_B = B^\frac{1}{2}$ and we can rewrite $L = A^\frac{1}{2}V_A^\top$ and $M = B^\frac{1}{2}U_B^\top$. 
For the singular value decomposition $USV^\top$ of $B^\frac{1}{2}A^\frac{1}{2}$, we get $A^\frac{1}{2}B^\frac{1}{2} = VSU^\top$ and $A^{-\frac{1}{2}}B^{-\frac{1}{2}} = V S^{-1} U^\top$. This allows us to express $C_\lambda$ in terms of $L$ and $M$, as
\begin{equation}
\label{eq:Clambda.2}
\begin{split}
    C_\lambda 
    &= \frac{1}{2}L\Big(4L^\top M M^\top L + \frac{\lambda^2}{4}I\Big)^{\frac{1}{2}} L^{-1} - \frac{\lambda}{4} I = L (V_A V) f_\lambda(S) (U_B U)^\top  M^\top.
\end{split}
\end{equation}
Also, notice that $M^\top L = U_B B^\frac{1}{2}A^\frac{1}{2}V_A^\top = U_B U S V^\top V_A^\top = (U_B U) S (V_A V)^\top $, so that $(U_B U) S (V_A V)^\top $ is the singular value decomposition of $M^\top L$. Therefore, by \Cref{thm.EAW_dT_sol}, the optimal coupling for $\AW_{2,\lambda}$ is given by $\pi_P$ defined in \eqref{eq:gauss_ansatz} with $P \coloneqq (V_A V) f_\lambda(S) (U_B U)^\top$. By \eqref{eq:Clambda.2}, $\pi_P$ is exactly the optimizer of $\W_{2,\lambda}$ defined in \eqref{eq:ewd.optimizer}.

When $T \geq 2$, by definition $\W_{2,\lambda} \leq \AW_{2,\lambda}$.
In the case of $d=1$, this yields the following entropic trace inequality:
\begin{equation*}
\begin{split}
    &\hspace{-1cm}\AW_{2,\lambda}(\mu, \nu) - \W_{2,\lambda}(\mu,\nu)\\
    &= 2\tr(S f_\lambda(S)) + \frac{\lambda}{2}\log\det(I - f_\lambda^2(S)) - 2\tr(N f_\lambda(N)) - \frac{\lambda}{2}\log\det(I - f^2_\lambda(N))\\
    &= \tr(g_{\lambda}(S)) - \tr(g_{\lambda}(N)) \geq 0,
\end{split}
\end{equation*}
where $g_{\lambda}(s) = 2sf_\lambda(s) + \frac{\lambda}{2}\log(1-f_\lambda^2(s))$ and $N = M^\top L$. 
In particular, when $\lambda = 0$, we recover $\tr(\calS(N)) \geq \tr(N)$, where $\calS(N)$ denotes the singular value matrix of $N$.
\subsection{On Monge optimizers of adapted Wasserstein distance} 
\label{subsec.Monge}
When $\mu$ and $\nu$ are non-degenerate Gaussians, the Monge map is given by the linear transformation
$x \mapsto b + A^{-\frac{1}{2}}(A^{\frac{1}{2}}B A^{\frac{1}{2}})^{\frac{1}{2}}A^{-\frac{1}{2}}(x-a)$, which induces the unique $\mathcal W_2$-optimal coupling. However, for $\AW_{2}$, optimal Gaussian couplings $\pi_P$ characterized in \Cref{thm.EAW_dT_sol} (ii) may not be Monge couplings, in the sense that they have no associated Monge maps.
Notice that $\pi_P$ is a Monge coupling if and only if there exists $H \in \R^{dT\times dT}$ such that
\[
\begin{bmatrix}
    I \\
    H \\
\end{bmatrix}
LL^\top
\begin{bmatrix}
    I & H^\top\\
\end{bmatrix} = \begin{bmatrix}
    L & 0\\
    0 & M\\
\end{bmatrix}\begin{bmatrix}
    I & P\\
    P^\top & I\\
\end{bmatrix}\begin{bmatrix}
    L^\top & 0\\
    0 & M^\top\\
\end{bmatrix}.
\]
After simple matrix multiplication, we see that the equality holds if and only if $H = MP^\top L^{-1}$  and $PP^\top = I$. Therefore, we conclude that 
\[
\text{$\pi_P$ is a Monge coupling $\Longleftrightarrow$ $P P^\top = I$,}
\]
in which case the associated Monge map is $T_{P}(x) = MP^\top L^{-1}x$ {\color{black} and $ MP^\top L^{-1}$ is block lower
triangular}. Moreover, since $P$ is block diagonal,
\[
PP^\top = I \Longleftrightarrow P_t P_t^\top = I, \, \forall t \leq T \Longleftrightarrow D_t D_t^\top = I, \,\forall t \leq T.
\]
Note that $D_t = I,  \forall t \leq T$ is always a valid choice of optimizer. So there always exists a Monge coupling.
\section{Proof of \Cref{thm.opt_eaot_gau}}
\label{sec.proofthm}
This section is devoted to proving \Cref{thm.opt_eaot_gau}. For this we need some preparatory results.
Recall that for non-degenerate Gaussian marginals, the entropic 2-Wasserstein distance has a unique, Gaussian optimal coupling, given in \eqref{eq:ewd.optimizer}. 
For general degenerate quadratic cost, we show that there still always exists a Gaussian optimizer.
\begin{lemma}
    \label{lem.ewd2}
    Let $\gamma_1 = \calN(m_1,\Sigma_1)$, $\gamma_2 = \calN(m_2,\Sigma_2)$ be non-degenerate Gaussians on $\R^d$, $\lambda \ge 0$ and $\Lambda \in \R^{d \times d}$. 
    Then 
    \begin{equation}
    \label{eq:lem:ewd2.1}
        \mathcal{V}_{\Lambda,\lambda}(\gamma_1,\gamma_2) \coloneqq \inf_{\pi \in \cpl(\gamma_1,\gamma_2)}\int | x - \Lambda y |^2 d\pi + \lambda \DKL(\pi|\mu \otimes \nu)
    \end{equation}
    admits a Gaussian optimizer, whose covariance matrix does not depend on $(m_1,m_2)$.
\end{lemma}
\begin{proof}
    Let $F\colon \R^{d\times d} \to \R^{d\times d}$ s.t. $F(x,y) = (x - m_1, y-m_2)$, $x,y \in \R^d$. Observe that, for all $\pi \in \cpl(\gamma_1, \gamma_2)$, we have $\DKL(\pi|\mu \otimes \nu) = \DKL(F_{\#}\pi|F_{\#}(\mu \otimes \nu))$ and 
    $$\int | x - \Lambda y |^2 d\pi = |m_1 - \Lambda m_2|^2 + \int | x - \Lambda y |^2 d(F_{\#}\pi).
    $$ 
    Since $F$ is invertible, we have $\mathcal{V}_{\Lambda,\lambda}(\gamma_1,\gamma_2) = |m_1 - \Lambda m_2|^2 + \mathcal{V}_{\Lambda,\lambda}(\calN(0,\Sigma_1),\calN(0,\Sigma_2))$. If $\pi^*$ is a Gaussian optimizer of $\mathcal{V}_{\Lambda,\lambda}(\calN(0,\Sigma_1),\calN(0,\Sigma_2))$, the covariance matrix of $\pi^*$ does not depend on $(m_1,m_2)$. Then $F^{-1}_\# \pi^*$ is a Gaussian optimizer of $\mathcal{V}_{\Lambda,\lambda}(\gamma_1,\gamma_2)$ and $F^{-1}_\# \pi^*$ has the same covariance matrix as $\pi^*$. Thus, w.l.g. we assume $m_1 = m_2 = 0$ and prove that $\mathcal{V}_{\Lambda,\lambda}(\gamma_1,\gamma_2)$ has a Gaussian optimizer.

    \medskip \noindent \emph{Case 1} ($\Lambda$ invertible):
    If $\Lambda$ is invertible, we set $f_{\Lambda}(x) = \Lambda x$, $g_{\Lambda}(x,y) = (x,\Lambda^{-1}y)$, $\tilde\gamma_2 = (f_\Lambda)_{\#}\gamma_2$, and find
    \begin{align}
        \mathcal V_{\Lambda,\lambda}(\gamma_1,\gamma_2)
        &= \inf_{\tilde{\pi} \in \cpl(\gamma_1,\tilde \gamma_2)}\int | x - z |^2 d\tilde{\pi} + \lambda \DKL\big((g_\Lambda)_{\#}\tilde{\pi}| (g_\Lambda)_{\#} (\gamma_1\otimes \tilde\gamma_2)\big)\nonumber\\
        &= \inf_{\tilde{\pi} \in \cpl(\gamma_1,\tilde\gamma_2)}\int | x - z |^2 d\tilde{\pi} + \lambda \DKL(\tilde{\pi}| \gamma_1\otimes \tilde\gamma_2),
        \label{eq:lem:ewd2.2}
    \end{align} 
    where the first equality holds by definition of the push-forward measure, and the second equality is due to the relative entropy being invariant under the push-forward by bijections.
    Since being Gaussian is preserved under linear transformations, we have that $\gamma_1, \tilde\gamma_2$ are both Gaussian, and thus the optimizer of \eqref{eq:lem:ewd2.2}, denoted by $\tilde{\pi}^{*}$, is also Gaussian and given by \eqref{eq:ewd.optimizer}. 
    Hence, $\pi^{*} \coloneqq (g_\Lambda)_{\#}\tilde{\pi}^{*}$ is a Gaussian optimizer of \eqref{eq:lem:ewd2.1}.

    \medskip \noindent \emph{Case 2} ($\Lambda$ not invertible):
    If $\Lambda$ is not invertible, let $\epsilon > 0$ be such that $\Lambda_\epsilon = \Lambda + \epsilon I$ is invertible.
    As in Case 1, we write $f_{\Lambda_\epsilon}(x) = \Lambda_\epsilon x$ and consider $\gamma_2^\epsilon = (f_{\Lambda_\epsilon})_{\#}\nu$.
    By Case 1, we have that \eqref{eq:lem:ewd2.1} admits a Gaussian optimizer between $\gamma_1$ and $\gamma_2^\epsilon$, which we denote by $\pi_\epsilon^{*}$.
    Let $\epsilon_n \searrow0$ be such that $\Lambda_{\epsilon_n}$ is invertible.
    Clearly, $\gamma_2^{\epsilon_n}$ converges to $\gamma_2$ in $\mathcal W_2$.
    By stability of entropic optimal transport, see e.g. \cite[Theorem 3.6]{eckstein2024computational}, the accumulation points of $(\pi_{\epsilon_n}^{*})_{n \in \N}$ are optimizers of \eqref{eq:lem:ewd2.1}. 
    Since the Gaussian distributions are a $\mathcal W_2$-closed subspace of $\mathcal P_2(\R^d)$, these accumulation points are again Gaussian, which proves that \eqref{eq:lem:ewd2.1} admits a Gaussian minimizer.    
\end{proof}

Next, we recall the dynamic programming principle for $\AW_{2,\lambda}(\mu,\nu)$ and the conditional law of Gaussian distributions in terms of their Cholesky decompositions. 

\begin{proposition}[Dynamic programming principle]
\label{prop.dpp}
Let $\mu,\nu\in\calP(\R^{dT})$. Set $V_T^{\mu,\nu}(x_{\bt{T}},y_{\bt{T}})=0$ and define, for all $t = 0,\ldots, T-1$,
\begin{align}
\label{eq:prop.dpp.1}
\begin{split}
    V_t^{\mu,\nu}(x_{\bt{t}},y_{\bt{t}}) = \inf_{\pi_{x_\bt{t},y_\bt{t}}^{t+1}\in \bccpl(\mu_{x_{\bt{t}}}^{t+1},\nu_{y_{\bt{t}}}^{t+1})}\int \Big[| x_{t+1} - y_{t+1}|^2 &+ \lambda \log\Big(\frac{d\pi_{x_\bt{t},y_\bt{t}}^{t+1}}{d(\mu\otimes\nu)_{x_\bt{t},y_\bt{t}}^{t+1}}\Big)\\
    &+ V^{\mu,\nu}_{t+1}(x_{\bt{t+1}},y_{\bt{t+1}})\Big]d\pi_{x_\bt{t},y_\bt{t}}^{t+1}.
\end{split}
\end{align}
Then $V_t^{\mu,\nu}(x_{\bt{t}},y_{\bt{t}}) = \AW^2_{2,\lambda}(\mu_{x_{\bt{t}}},\nu_{y_{\bt{t}}})$ for all $t = 0,\dots,T-1$; in particular $V_0^{\mu,\nu} = \AW^2_{2,\lambda}(\mu,\nu)$.
Moreover, let $\pi_{x_\bt{t},y_\bt{t}}^{t+1}$ be optimizers of $V_t^{\mu,\nu}(x_{\bt{t}},y_{\bt{t}})$, $t=0,\ldots,T-1$. Then $\pi_{x_\bt{t},y_\bt{t}}(dx_{\btp{t}}, dy_{\btp{t}}) \coloneqq \prod_{s=0}^{t}\pi_{x_\bt{s},y_\bt{s}}^{s+1}(dx_{s+1}, dy_{s+1})$ is an optimizer of $\AW_{2,\lambda}(\mu_{x_{\bt{t}}},\nu_{y_{\bt{t}}})$, $t= 0,\dots,T-1$; in particular, $\pi(dx,dy) \coloneqq \prod_{t=0}^{T-1}\pi_{x_\bt{t},y_\bt{t}}^{t+1}(dx_{t+1},dy_{t+1})$ is an optimizer of $\AW_{2,\lambda}(\mu,\nu)$.
\end{proposition}
\begin{proof}
    It follows from the separability of quadratic cost and log-likelihood; see \cite{eckstein2024computational,pichler2022nested}.
\end{proof}

\begin{lemma}
\label{lem.cond_law_gau}
Let $\mu = \mathcal N(a,A)$ be a non-degenerate Gaussian on $\R^{dT}$,  whose covariance matrix has Cholesky decomposition $A = LL^\top$.
Then, for all $t=1,\ldots,T-1$ and $x_{\bt{t}}\in\R^{dt}$,
\begin{equation}
\label{eq:lem.cond_law_gau.1}
        \mu_{x_{\bt{t}}} = \calN(a_{\mathbf{t'}} + L_{\btp{t},\bt{t}}L^{-1}_{\bt{t},\bt{t}}(x_{\bt{t}}-a_{\bt{t}}), L_{\btp{t},\btp{t}}L_{\btp{t},\btp{t}}^\top).
\end{equation}
\end{lemma}
\begin{proof}
    See \cite[Lemma 4.3]{gunasingam2024adapted}.
\end{proof}

\begin{remark}
The variance of $\mu_{x_{\bt{t}}}$ only depends on $t$ and does not depend on $x_{\bt{t}}$.
\end{remark}

Finally, we prove \Cref{thm.opt_eaot_gau} by back-propagating the Gaussianity of optimal couplings through the dynamic programming principle. 

\begin{proof}[Proof of \Cref{thm.opt_eaot_gau}]
W.l.g.\ we assume $a = b = 0$ and proceed to prove the assertion backwards in time.

\medskip \noindent \emph{Step 1:} For $t=T-1$, by Proposition~\ref{prop.dpp}, we have
\begin{equation}
\label{eq:thm.opt_eaot_gau.1}
    V_{T-1}^{\mu,\nu}(x_{\bt{T-1}},y_{\bt{T-1}}) = \AW^2_{2,\lambda}(\mu_{x_{\bt{T-1}}},\nu_{y_{\bt{T-1}}}) = \W^2_{2,\lambda}(\mu_{x_{\bt{T-1}}},\nu_{y_{\bt{T-1}}}).
\end{equation}
Notice that, for every $(x_{\bt{T-1}},y_{\bt{T-1}}) \in \R^{d(T-1)} \times \R^{d(T-1)}$, $\mu_{x_{\bt{T-1}}}$ and $\nu_{y_{\bt{T-1}}}$ are both non-degenerate Gaussian by Lemma~\ref{lem.cond_law_gau} and the covariance matrices are independent of $(x_{\bt{T-1}},y_{\bt{T-1}})$. Thus, by Lemma~\ref{lem.ewd2}, $\AW_{2,\lambda}(\mu_{x_{\bt{T-1}}},\nu_{y_{\bt{T-1}}})$ has a Gaussian optimal coupling, which we denote by $\pi_{x_\bt{T-1},y_\bt{T-1}}$, whose covariance matrix is independent of $(x_\bt{T-1},y_\bt{T-1})$.

\medskip  \noindent\emph{Step 2: }Suppose that, for some $t \in \{1,\dots,T-1\}$ and every $(x_{\bt{t}},y_{\bt{t}}) \in \R^{dt} \times \R^{dt}$, $\AW_{2,\lambda}(\mu_{x_{\bt{t}}},\nu_{y_{\bt{t}}})$ has a Gaussian optimal coupling $\pi_{x_{\bt{t}},y_{\bt{t}}}$, whose covariance matrix is independent of $(x_{\bt{t}},y_{\bt{t}})$. 
This means that there exists a matrix $C_t 
\in \R^{d(T-t)\times d(T-t)}$ s.t.\
\begin{equation*}
    \pi_{x_{\bt{t}},y_{\bt{t}}} = \calN\Big(
    \begin{bmatrix}
        L_{\btp{t},\bt{t}}L^{-1}_{\bt{t},\bt{t}}x_{\bt{t}} \\
        M_{\btp{t},\bt{t}}M^{-1}_{\bt{t},\bt{t}}y_{\bt{t}}
    \end{bmatrix},
    \begin{bmatrix}
        L_{\btp{t},\btp{t}}L_{\btp{t},\btp{t}}^\top & C_t \\
        C_t^\top & M_{\btp{t},\btp{t}}M_{\btp{t},\btp{t}}^\top
    \end{bmatrix}\Big)
    \eqqcolon
    \calN\Big(
    \begin{bmatrix}
        \hat{a}_t \\
        \hat{b}_t
    \end{bmatrix},
    \begin{bmatrix}
        \hat{A}_t  & C_t \\
        C_t^\top & \hat{B}_t 
    \end{bmatrix}\Big).
\end{equation*}
By Proposition~\ref{prop.dpp} and plugging $\pi_{x_{\bt{t}},y_{\bt{t}}}$ into the cost functional in \eqref{eq:AW.reg}, we get
\begin{equation}
    \label{eq:thm.opt_eaot_gau.2}
    \begin{split}
     V_t^{\mu,\nu}(x_{\bt{t}},y_{\bt{t}}) = \AW^2_{2,\lambda}(\mu_{x_{\bt{t}}},\nu_{y_{\bt{t}}}) &= |\hat{a}_t-\hat{b}_t|^2 + \AW^2_{2,\lambda}(\calN(0,\hat{A}_t),\calN(0,\hat{B}_t)),
    \end{split}
\end{equation}
where $\AW^2_{2,\lambda}(\calN(0,\hat{A}_t),\calN(0,\hat{B}_t)) \eqqcolon R_t$ does not depend on $(x_{\bt{t}}, y_{\bt{t}})$.
Now, we plug \eqref{eq:thm.opt_eaot_gau.2} into \eqref{eq:prop.dpp.1} in Proposition~\ref{prop.dpp}, with $t=t-1$, and get
\begin{align*}
   V_{t-1}^{\mu,\nu}(x_{\bt{t-1}},y_{\bt{t-1}}) = \inf_{\pi^t_{x_{\bt{t-1}},y_{\bt{t-1}}}\in \cpl(\mu^t_{x_{\bt{t-1}}},\nu^t_{y_{\bt{t-1}}})}&\int \Big[| x_{t} - y_{t}|^2 + | L_{\btp{t},\bt{t}}L^{-1}_{\bt{t},\bt{t}}x_{\bt{t}} - M_{\btp{t},\bt{t}}M^{-1}_{\bt{t},\bt{t}}y_{\bt{t}} |^2\\ &+ \lambda \log(\frac{d\pi_{x_\bt{t-1},y_\bt{t-1}}^{t}}{d(\mu\otimes\nu)_{x_\bt{t-1},y_\bt{t-1}}^{t}}) + R_t\Big]d\pi_{x_{\bt{t-1}},y_{\bt{t-1}}}^t,
\end{align*}
which after expanding the quadratic terms and recombining them, is equivalent (up to a constant) to 
\begin{equation}
\label{eq:thm.opt_eaot_gau.3}
\begin{split}
    \inf_{\pi_{x_{\bt{t-1}},y_{\bt{t-1}}}^t\in \cpl(\mu^t_{x_{\bt{t-1}}},\nu^t_{y_{\bt{t-1}}})}\int \Big[| x_{\bt{t}} - \Lambda_t y_{\bt{t}}|^2 + \lambda \log(\frac{d\pi_{x_\bt{t-1},y_\bt{t-1}}^t}{d(\mu\otimes\nu)_{x_\bt{t-1},y_\bt{t-1}}^t})\Big]d\pi_{x_{\bt{t-1}},y_{\bt{t-1}}}^t,
\end{split}
\end{equation}
where $\Lambda_t = \bI + L^{-\top}_{\bt{t},\bt{t}} L^{\top}_{\btp{t},\bt{t}}M_{\btp{t},\bt{t}}M^{-1}_{\bt{t},\bt{t}}$. Notice that, by Lemma~\ref{lem.cond_law_gau}, $\mu_{x_{\bt{t-1}}}^t$ and $\nu_{y_{\bt{t-1}}}^t$ are Gaussian with covariance matrices independent of $(x_{\bt{t-1}}, y_{\bt{t-1}})$. Thus, by Lemma~\ref{lem.ewd2}, \eqref{eq:thm.opt_eaot_gau.3} has a Gaussian optimal coupling $\pi^{t}_{x_{\bt{t-1}},y_{\bt{t-1}}}$, whose covariance matrix is also independent of $(x_{\bt{t-1}}, y_{\bt{t-1}})$.
Therefore, by Proposition~\ref{prop.dpp}, 
$\pi_{x_{\bt{t-1}},y_{\bt{t-1}}}(dx_{t:T},dy_{t:T}) \coloneqq \pi^{t}_{x_{\bt{t-1}},y_{\bt{t-1}}}(dx_{t},dy_{t})\pi_{x_{\bt{t}},y_{\bt{t}}}(dx_{\btp{t}},dy_{\btp{t}})$ 
is a Gaussian optimal coupling of $\AW_{2,\lambda}(\mu_{x_{\bt{t-1}}},\nu_{y_{\bt{t-1}}})$ and the covariance matrix of $\pi_{x_{\bt{t-1}},y_{\bt{t-1}}}$ is independent of $(x_\bt{t-1},y_\bt{t-1})$.

\medskip \noindent \emph{Step 3:} By induction, for all $t \in \{0,\dots, T-1\}$, $V_t^{\mu,\nu}(x_{\bt{t}},y_{\bt{t}})$ has a Gaussian optimal coupling; in particular, there is a Gaussian optimal coupling for $V_0^{\mu,\nu} = \AW^2_{2,\lambda}(\mu,\nu)$, which yields the assertion.
\end{proof}

\section{Examples}
\label{sec:example}
In this section we compare optimal couplings for different optimal transport distances, in particular for $\AW_{2}$ and $\AW_{2,\lambda}$, we consider the McCann displacement interpolation (\cite{mccann1997convexity}): $(\mu_t)_{t\in[0,1]}$ s.t. $\mu_t = ((1-t)X +tY)_{\#}\sP, t \in [0,1]$, where $(X,Y)\sim \pi^*$ and $\pi^*$ is the optimal coupling. We showcase that $\AW_{2}$-optimal couplings may be non-unique, non-Gaussian, and that, even if the coupling is Gaussian, the $\AW_{2}$ displacement interpolation may degenerate. On the other hand, the $\AW_{2,\lambda}$-optimal coupling is always unique, Gaussian, and the $\AW_{2,\lambda}$ displacement interpolation remains non-degenerate, as long as $\lambda > 0$. 

\begin{example}[Non-unique optimal coupling]
\label{ex:2}
Let $d=1$, $T=2$, \( \mu = \mathcal{N}(a, A) \) and \( \nu = \mathcal{N}(b, B) \), with
\[
    a = \begin{bmatrix}
    0 \\
    0
    \end{bmatrix},\quad\!\!
    b = \begin{bmatrix}
    6 \\
    -6
    \end{bmatrix},\quad\!\!
    A = \begin{bmatrix}
    1 & 2 \\
    2 & 5
    \end{bmatrix},\quad\!\!
    B = \begin{bmatrix}
    1 & -0.5 \\
    -0.5 & 1.25
    \end{bmatrix}, \quad\!\!
    L = \begin{bmatrix}
    1 & 0 \\
    2 & 1
    \end{bmatrix},\quad\!\!
    M = \begin{bmatrix}
    1 & 0 \\
    -0.5 & 1
    \end{bmatrix},
\]
so that
$A = LL^\top$ and $B = MM^\top$. 
Observe that $L^\top M = \begin{bmatrix}
    0 & 2 \\
    -0.5 & 1
    \end{bmatrix}$
and $(L^\top M)_{1,1} = 0$. 
Then, the $\AW_{2}$-optimal coupling $\pi^*
_{\AW_{2}}$ is not unique, since there is one degree of freedom to choose optimal Gaussian couplings from.
Indeed,
\begin{equation*}
    \pi^*_{\AW_{2}} =
    \mathcal N\Big(
    \begin{bmatrix}
        a \\
        b
    \end{bmatrix},
    \begin{bmatrix}
        LL^\top & LP^*_\lambda M^\top \\
        (LP^*_\lambda M^\top)^\top & MM^\top
    \end{bmatrix}\Big), \quad \text{with\; }
    P^*_\lambda = \begin{bmatrix}
        \rho & 0 \\
        0 & 1
    \end{bmatrix},
\end{equation*}
is an optimal coupling as long as $\rho \in [-1,1]$.
In particular, choosing $\rho = 1, 0, -1$ corresponds to the monotone, independent, and anti-monotone coupling of the first marginals, respectively. These different optimal couplings can be observed in Figure~\ref{fig:cpl_rho}. 
\begin{figure}[H]
\centering
\includegraphics[width=\textwidth]{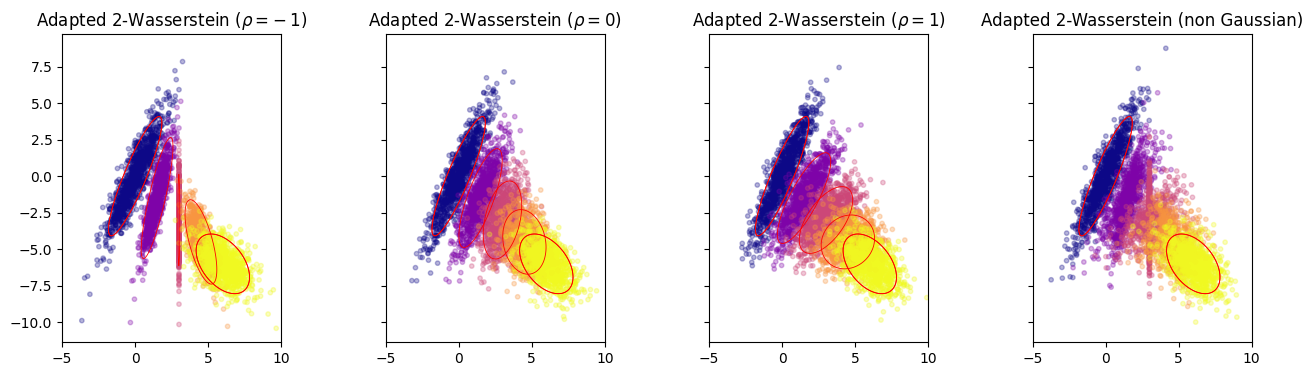}
\caption{Different optimal couplings of $\AW_{2}$}
\label{fig:cpl_rho}
\end{figure}
In contrast to $\AW_{2}$-optimal couplings, the $\AW_{2,\lambda}$-optimal coupling $\pi^*
_{\AW_{2,\lambda}}$ is unique if $\lambda > 0$. Moreover, as $\lambda \to 0$, this converges to  $\pi^*_{\AW_{2}}$, corresponding to $\rho = 0$, which is the coupling with maximal entropy among all optimal couplings for $\AW_2$. 
Therefore, $\pi^*
_{\AW_{2,\lambda}}$ can be seen as a robust alternative to $\pi^*
_{\AW_{2}}$ when the optimizers are not unique. 

Finally, we show that there may be non-Gaussian optimizers. Intuitively, this can be constructed by considering a bi-causal coupling as a mixture of optimal couplings $\pi^*
_{\AW_{2}}$ corresponding to $\rho = -1$ and $\rho = 0$. Let $Z_1, Z_2$ be standard Gaussian random variables on $\R$, and $Z_W$ a coin flip, i.e., a Bernoulli random variable taking values $1$ and $-1$ with equal probability. Let $Z_1, Z_2, Z_W$ be all independent and define 
\[
Z^X = \begin{bmatrix}
    Z_1\\
    Z_2
\end{bmatrix},\quad 
Z^Y = 
\begin{bmatrix}
    W Z_1\\
    Z_2
\end{bmatrix},\quad
\begin{bmatrix}
    X \\
    Y
\end{bmatrix} = \begin{bmatrix}
    a \\
    b
\end{bmatrix} + 
\begin{bmatrix}
    L & 0 \\
    0 & M
\end{bmatrix}\begin{bmatrix}
    Z^X \\
    Z^Y
\end{bmatrix}.
\]
Then the distribution of $(X,Y)$ is also a $\AW_{2}$ optimal coupling, but clearly not a Gaussian one, since $Z_1 + WZ_1$ is a mixture of a Gaussian and a Dirac measure; see Figure~\ref{fig:cpl_rho}.
\end{example}

\begin{example}[Degenerate displacement interpolation]
\label{ex:1}Let $d=1$, $T=2$, and consider the parameters
\[
a = \begin{bmatrix}
0 \\
0
\end{bmatrix},\quad
b = \begin{bmatrix}
6 \\
-6
\end{bmatrix},\quad
A = \begin{bmatrix}
1 & 2 \\
2 & 5
\end{bmatrix},\quad
B = \begin{bmatrix}
1 & -1 \\
-1 & 2
\end{bmatrix},\quad 
L = \begin{bmatrix}
1 & 0 \\
2 & 1
\end{bmatrix},\quad 
M = \begin{bmatrix}
1 & 0 \\
-1 & 1
\end{bmatrix},
\] 
so that
 \( A = LL^\top \) and \( B = MM^\top \).
For \( \mu = \mathcal{N}(a, A) \) and \( \nu = \mathcal{N}(b, B) \), we compute the optimal couplings under different distances, which are $\W_2$ (with optimal coupling $\pi^*_{\W_2}$),  $\W_{2,\lambda}$ (with optimal coupling $\pi^*_{\W_{2,\lambda}}$) for $\lambda = 1$, $\AW_{2}$ (with optimal coupling $\pi^*_{\AW_2}$), and $\AW_{2,\lambda}$ (with optimal coupling $\pi^*_{\AW_{2,\lambda}}$) for $\lambda = 1$, which are all unique in this example; see \Cref{fig:cpl}.
\begin{figure}[H]
\centering
\includegraphics[width=\textwidth]
{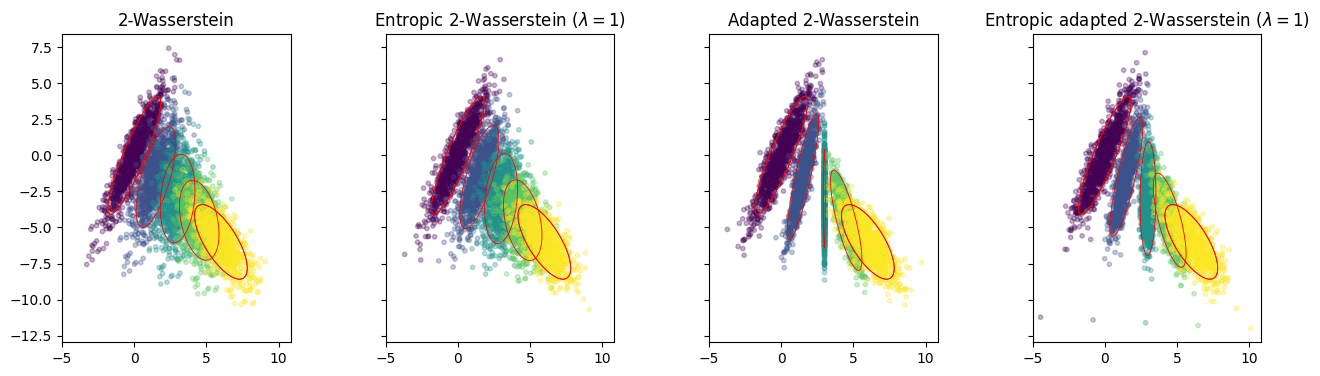}
\caption{Optimal couplings for different distances}
\label{fig:cpl}
\end{figure}
We see from \Cref{fig:cpl} that the $\AW_{2}$-displacement interpolation at $t = 0.5$ is degenerate.
This showcases that the $\AW_{2}$-displacement interpolation can degenerate even with non-degenerate marginals (cf. Example~6.2 in \cite{gunasingam2024adapted}). However, the $\W_{2}$, $\W_{2,\lambda}$, and $\AW_{2,\lambda}$-displacement interpolations are always non-degenerate Gaussian when $\lambda > 0$. Therefore, we can see $\pi^*
_{\AW_{2,\lambda}}$ as the non-degenerate approximation of $\pi^*
_{\AW_{2}}$ when $\pi^*
_{\AW_{2}}$ is unique. 
\end{example}
\begin{example}[Multi-dimension]
\label{ex:3}
Let $d=2$, $T=2$, \( \mu = \mathcal{N}(a, A) \) and \( \nu = \mathcal{N}(b, B) \), with $a=b=0\in \R^4$,
\[
A = \begin{bmatrix}
I  & L_{2,1}^\top \\
L_{2,1} & A_{2,2} 
\end{bmatrix},\quad
B = \begin{bmatrix}
I  & I \\
I & 2I 
\end{bmatrix},\quad
L = \begin{bmatrix}
I  & 0 \\
L_{2,1} & I 
\end{bmatrix}, \quad 
M = \begin{bmatrix}
I & 0 \\
I & I
\end{bmatrix},\quad
L_{2,1} = \frac{\sqrt{2}}{10}\begin{bmatrix}
7 & -1 \\
1 & 7
\end{bmatrix} - I,
\]     
and $A_{2,2} = L_{2,1}L_{2,1}^\top + I = \begin{bmatrix}
3 - \frac{7\sqrt{2}}{5} & 0 \\
0 & 3 - \frac{7\sqrt{2}}{5}
\end{bmatrix}$. Then $(M^\top L)_{2,2} = I = U_2 S_2 V_2^\top$ and $(M^\top L)_{1,1} = I + M_{2,1}^\top L_{2,1} =  \frac{\sqrt{2}}{10}\begin{bmatrix}
7 & -1 \\
1 & 7
\end{bmatrix} = U_1 S_1 V_1^\top,
$ where $
U_1 = \frac{\sqrt{2}}{2}
\begin{bmatrix}
1 & 1 \\
-1 & 1 
\end{bmatrix}$, $V_1 = \frac{1}{5}
\begin{bmatrix}
3 & 4 \\
-4 & 3 
\end{bmatrix}$, $S_1 = U_2 = V_2 = S_2 = I$. Therefore $S = I$, $D_\lambda = \diag([f_{\lambda}(1)]_{t=1}^{4})$ and
\begin{equation}
    \begin{split}
        \AW_{2,\lambda}^2(\mu,\nu)
        &= |a-b|^2 + \tr(A + B) - 2\tr(D_{\lambda} S) - \frac{\lambda}{2}\log\det(I - D_{\lambda}^2)\\
        &= 0 + 8-\frac{14\sqrt{2}}{5} + 6 - 8f_\lambda(1) - 2\lambda\log(1-f_{\lambda}(1)^2).
    \end{split}
\end{equation}
In particular, when $\lambda = 0$, $\AW_2^2(\mu,\nu) = 6$. Since $S$ is invertible, $\pi^* = \pi_P$ defined by \eqref{eq:thm.AW.reg.optimizer} is the unique optimizer of $\AW_{2}(\mu,\nu)$, with $P = \begin{bmatrix}
    P_1 & 0 \\
    0 & I
\end{bmatrix}$, $ P_1 = V_1^\top U_1 =
\frac{\sqrt{2}}{10}\begin{bmatrix}
7 & -1 \\
1 & 7
\end{bmatrix}$. This shows that, when $d > 1$, $P$ is in general block-diagonal but not diagonal.
\end{example}

\printbibliography
\end{document}